\documentclass[3p]{elsarticle}
\usepackage[english]{babel}
\usepackage{float} 
\usepackage[T1]{fontenc}
\usepackage[utf8]{inputenc}
\usepackage{graphicx}
\usepackage{amsmath}
\usepackage{amsfonts}
\usepackage{amssymb}

\newtheorem{lem}{Lemma}
\newtheorem{prop}{Proposition}
\newtheorem{theorem}{Theorem}

\newenvironment{proof}
{\begin{sloppypar}\noindent{\sf Proof~: }}
{\hspace*{\fill}\eop
\medskip
\end{sloppypar}}

\newcommand{\CC}{\mathbb{C}}
\newcommand{\NN}{\mathbb{N}}
\newcommand{\RR}{\mathbb{R}}

\newcommand{\LL}{\mathop{\longrightarrow}}
\newcommand{\integ}[4]{\displaystyle{\int_{#1}^{#2}\!{#3}\,d{#4}}}
\def\eop{\hbox{{\vrule height7pt width3pt depth0pt}}}

\title{A Product Integration type Method for solving Nonlinear Integral Equations in $L^1$ }
\author[1,2]{L. Grammont}
\author[1,3]{H. Kaboul}
\author[1,4]{M. Ahues}
\address[1]{Institut Camille Jordan,  UMR 5208, CNRS-Universit\'e de Lyon\\23 rue du Dr Paul Michelon, 42023 Saint-\'Etienne Cedex 2, France.}
\address[2]{laurence.grammont@univ-st-etienne.fr}
\address[3]{hanane.kaboul@univ-st-etienne.fr}
\address[4]{mario.ahues@univ-st-etienne.fr}

\date{}

\begin{document}
\maketitle

\addcontentsline{toc}{section}{Notation}
\section*{Abstract}
This paper deals with nonlinear Fredholm integral equations of the second kind. We study the case of a weakly singular kernel and we set the problem in the space $L^1([a,b],\CC)$. As numerical method, we extend the product integration scheme from $C^0([a,b],\CC)$ to $L^1([a,b],\CC)$.

\bigskip

\noindent
\textbf{\textbf{Keywords:}} Fredholm integral equation, product integration method, nonlinear equation.

\section{Introduction}
In this paper, we consider  the fixed point problem 
\begin{equation}\label{eq1}
\mbox{Find } \varphi:\quad U(\varphi)=\varphi,
\end{equation}
where $U$ is of the form:
\begin{equation}\label{opr1}
U(x):=K(x)-y \mbox{ for all } x\in \Omega.
\end{equation}
The domain $\Omega$ of $U$ is in $L^1([a,b],\CC)$ and $y\in L^1([a,b],\CC)$.   

The operator $K$ is of the following form:
$$
K(x)(s):=\integ{a}{b}{ H(s,t)L(s,t)N(x(t))}{t}  \mbox{ for all } x\in \Omega,
$$
and $N:\RR\longrightarrow \CC$ is twice Fréchet-differentiable and may be nonlinear.

This kind of equations are usually treated in the space of continuous functions $C^0([a,b],\CC)$. In \cite{atkinson2}, Atkinson gives a survey about the main numerical methods which can be applied to such integral equations of the second kind  (projection method, iterated projection method, Galerkin’s method,  Collocation method, Nystr\"om method, discrete Galerkin method...) (see also\cite{kra1}). The approximate solution $\varphi_n$ of (\ref{eq1}) is the solution of an approximate equation of the form~: 
\begin{eqnarray}\label{eq1-approx}
\mbox{Find }\varphi_n\in L^1([a,b],\CC):\quad U_n(\varphi_n)=\varphi_n,\end{eqnarray}
where 
$$
U_n(x)=K_n(x)-y_n,
$$
$K_n$ being  an approximation of the operator $K$ and  $y_n$ an approximation of $y$. For the classical projection method, $K_n=\pi_n K \pi_n$, where  $\pi_n$ is a projection onto a finite dimensional space, and $y_n=\pi_n y$. For the Kantorovich projection method, $K_n=\pi_n K $ and $y_n= y$. For the Iterated projection method, $K_n= K \pi_n$ and $y_n=y$. For the Nystr\"om method, $K_n$  is provided by a numerical quadrature of the integral operator $K$. In \cite{atkinson2}, the Banach space in which the solution $\varphi $ is found is the space of continuous function or eventually  $L^2$ and the kernel is smooth. When solving numerically weakly singular equations, one is required to evaluate large number of weakly singular integrals. In this case, when the integral operator is still compact,  the technique of product integration methods appears to have been a popular choice to approximate such integrals (see \cite{anselone1}, \cite{atkinson}, \cite{atkinson1}, \cite{atkinson2}). This method requires the unknown to be smooth. The product integration method consists in performing a linear interpolation of the smooth part of the kernel times the unknown.  The product integration method is used to treat linear or nonlinear Volterra equation or Fredholm equation of the second kind  and for each type of equation, different kernels are studied: logarithmic singular  kernels in \cite{diogo}, kernels with a fixed Cauchy singularity coming from scattering theory in \cite{bertram}, Abel's kernel coming from the theory of fluidity and heat transfer between solid and gases in \cite{cameron}. Most of these papers need  evaluations of the unknown at the nodes and also continuity of the exact solution.     

We consider the numerical treatment of equation (\ref{eq1}) when 
$\varphi$ can not be evaluated at each point. In this situation we tackle the problem of  applying a product integration type method. We propose a kind  of hybrid method between a product integration method and a iterated projection method for which the general theory of Anselone (see \cite{anselone}) and Ansorge (see \cite{ansorge}) can be applied through collectively compact convergence theory. 

In Section 2, we recall the framework of the paper and the results needed to prove our main result.   In Section 3,  we present our main result. We prove the existence, the uniqueness and the convergence of our method. Section 4 is devoted to the numerical implementation of the method and an illustration of our theoretical results.    

\section{General framework}
To prove the existence and the uniqueness of the approximate solution, we use a general result of Atkinson (see \cite{atkinson1}, Theorem 4 p 804) recalled in this paper (see Theorem \ref{atkinson}). To apply this theorem, we need to check if the assumptions are satisfied in our case. As the framework of our problem is the space $L^1([a,b],\CC)$, to prove the compactness of the operators and the collectivelly compactness of the sequence of approximate operators, we will use the Kolmogorov-Riesz-Fréchet theorem recalled in this paper too (see Theorem \ref{KRF}).\\
Here, we assume that $U_n(x)$ is of the form $K_n(x)-y$ ($y_n=y$). 

\bigskip

\textbf{Hypotheses:}

\begin{itemize}
\item[(H1)] $\varphi$ denotes a fixed point of $U$. $X$ is a complex Banach space, $\Omega_r(\varphi)$ is the open ball centered at $\varphi$ and with radius $r>0$ of the space $L^1([a,b],\CC)$, $U$ and $U_n$, for $n\geq 1$, are completely continuous possibly nonlinear operators from $\Omega_r(\varphi)$ into $X$.
\item[(H2)] $(U_n)_{n\geq 1}$ is a collectively compact sequence.
\item[(H3)] $(U_n)_{n\geq 1}$ is pointwise convergent to $U$ on $\Omega_r(\varphi)$.
\item[(H4)] There exists  $r_\varphi>0$, such that $U$ and $U_n$, for $n\geq 1$, are twice Fréchet differentiable on $\Omega_{r_\varphi}(\varphi)\subset \Omega_r(\varphi)$, and there exists a least upper bound $M(\varphi,r)$ such that 
$$
\max_{x  \in \Omega_r(\varphi)  }\{\Vert U''(x)\Vert,\Vert U_n''(x)\Vert\}\leq M(\varphi,r).
$$
\end{itemize}
\begin{theorem}\label{atkinson}
Assume that {\rm(H1)} to {\rm(H4)} are satisfied and that $1$ is not an eigenvalue of $U'(\varphi)$. Then $\varphi$ is an isolated fixed point of $U$. Moreover, there is $\epsilon$ in $]0,r_\varphi[$ and $n_{\epsilon}>0$ such that, for all  $n\geq n_\epsilon$, $U_n$ has a unique fixed point $\varphi_n$ in $\Omega_\epsilon(\varphi)$. Also, there is a constant $\gamma>0$ such that 
\begin{equation}\label{estatkin}
\Vert \varphi-\varphi_n\Vert \leq \gamma\Vert U(\varphi)-U_n(\varphi)\Vert \qquad \mbox{ for } n\geq n_\epsilon.
\end{equation}
\end{theorem}
\begin{proof}
See Theorem 4 in \cite{atkinson1}.
\end{proof}

To prove that the assumptions (H1) and (H2) are satisfied in our case, we use the Kolmogorov-Riesz-Fréchet theorem, recalled here below.

\begin{theorem}(Kolmogorov-Riesz-Fréchet)\label{KRF}
Let $F$ be a bounded set in $L^p(\RR^q,\CC)$, $1\leq p\leq+ \infty$. If
$$
\lim_{\Vert h\Vert \to0}\Vert\tau_hf-f\Vert_p=0
$$
uniformly in $f\in F$, where 
$$
\tau_h f(\cdot):=f(\cdot+h),
$$
then the closure of $F|_\Omega$ is compact in $L^p(\Omega,\CC)$ for any measurable set $\Omega\in \RR^p$ with finite measure.
\end{theorem}

In our error estimation analysis, we need to define the following quantities~:

The oscillation of a function $x$ in  $L^1([a,b],\CC)$, relatively to a parameter $h$, is defined by 
\begin{equation}\label{defw1}
w_1(x,h):=\sup_{|u| \in [0,|h|]}\int_a^b| \widetilde{x}(v+u)-\widetilde{x}(v)|dv,
\end{equation}
where 
$$ 
\tilde{x}(t) := \left \{
\begin{array}{ll}
x(t)&\textrm{for $t\in [a,b]$},\\
0 &\textrm{for $t\notin [a,b]$}.
\end{array} \right.
$$
 
The modulus of continuity of a continuous function  on $[a,b]\times [a,b]$,  relatively to a parameter $h$, is defined by
\begin{equation}\label{defw2}
w_2(f,h):=\sup_{u,v \in [a,b]^2 , \Vert u-v\Vert \leq |h|} |f(u)-f(v)|.\end{equation}
 
\begin{lem}
For all $x$ in $L^1([a,b],\CC)$,
$$
\lim_{h\to 0}w_1(x,h)=0.
$$
For all $f$ in  $C^0([a,b]^2,\CC)$, 
$$
\lim_{h\to 0}w_2(f,h)=0.
$$
\end{lem}

\begin{proof}
See \cite{ahues1}.
\end{proof}

\section{Product integration in $L^1$}

Let $\pi_n$ be the projection defined with a uniform grid as follows:
\begin{eqnarray*}
\forall i=0,\dots,n, \qquad t_{n,i}&:=&a+ih_n,\\\\
h_n&:=&\frac{b-a}{n}.
\end{eqnarray*}
For $i=1,\dots,n$,
$$
\forall x\in L^1([a,b],\CC), \pi_n(x)(t):=\frac{1}{h_n}\integ{t_{n,i}}{t_{n,i-1}}{x(v)}{v}=c_{n,i}, \ t\in [t_{n,i-1},t_{n,i}].
$$
It is obvious that $\Vert \pi_n h \Vert \leq \Vert h \Vert$ and $\Vert \pi_n   \Vert=1$. We also have
$$
\pi_n\LL^p  I,
$$
where $\LL\limits^{p}$ denotes the  pointwise convergence and $I$ the identity operator. In fact, (see \cite{ahues1}),
\begin{equation}\label{pin}
\Vert \pi_n(x)-x\Vert \leq 2w_1(x,h_n)
\end{equation}
To approximate problem (\ref{eq1}), we define the operator
$$
K_n(x)(s):=\integ{a}{b}{H(s,t)[L(s,t)]_nN(\pi_n(x)(t))}{t},
$$
where, $\forall s\in [a,b],\forall i=1,\dots,n$:
\begin{align*}
[L(s,t)]_n&:=\frac{1}{h_n}\left( (t_{n,i}-t)L(s,t_{n,i-1})+(t-t_{n,i-1})L(s,t_{n,i})\right)
\end{align*}
for  $t\in[t_{n,i-1},t_{n,i}]$.\\
Consequently, the approximate operator $U_n$ will be defined by 
\begin{eqnarray}
U_n(x):=K_n(x)-y.
\end{eqnarray}

\noindent{\bf Notations:}

\bigskip

\noindent
$\Vert\cdot\Vert$  denotes the norm of the underlying vector space, whatever  it may be. As usual $K'$  denotes the first order  Fréchet-derivative of $K$, and $K''$ its second order Fréchet-derivative.

Let us define the following operator $A_0$: $\forall x\in \Omega_r(\varphi)$, $\forall s\in [a,b]$,
\begin{eqnarray*}
A_0(x) &:& s\mapsto\integ{a}{b}{|H(s,t)||N(x(t))|}{t},
\end{eqnarray*}
provided that the integral exists. 

We make the following assumptions on $L$, $H$ and $N$:

\begin{itemize}
\item [(P1)] $L\in C^0([a,b]^2,\CC)$ and 
$$
c_L:=\max_{s,t\in [a,b]}|L(s,t)|.
$$
\item[(P2)] There exists $r>0$ such that, $\Omega_r(\varphi)\subset \Omega$, and there exist $m_0>0$, $M_0>0$, $M_1>0$, $M_2>0$, $C_1>0$, $C_2>0$, $M>0$ and $C>0$  such that:
\item[(P2.1)]  $\forall x\in \Omega_r(\varphi), A_0(x) \in L^1$ and $\forall n \in \NN, A_0(\pi_n\varphi)\in L^1$ and
\begin{eqnarray*} 
\sup_{x\in \Omega_r(\varphi)}\|A_0(x)\|&\le&M_0,\\
\sup_{x\in \Omega_r(\varphi)}\|A_0(\pi_n(x))\|&\le&m_0.
\end{eqnarray*}
\item[(P2.2)] $K$  is twice Fr\'echet-differentiable and 
\begin{eqnarray*}
\sup_{x\in \Omega_r(\varphi)}\Vert K'(x)\Vert&\le&M_1,\\
\sup_{x\in \Omega_r(\varphi)}\Vert K''(x)\Vert&\le&M_2.
\end{eqnarray*}
\item[(P2.3)]  For $n$ large enough,
\begin{eqnarray*}
\sup_{n \in \NN}\Vert K_n'(\varphi)\Vert&\le&C_1,\\
\sup_{x\in \Omega_r(\varphi)}
\Vert K_n''(x)\Vert&\le&C_2.
\end{eqnarray*}
\item[(P2.4)]   
\begin{eqnarray*}
\sup_{x\in \Omega_r(\varphi)}\integ{a}{b}{|N(\pi_n(x)(t))|}{t} \leq M.
\end{eqnarray*}
\begin{eqnarray*}
\sup_{x\in \Omega_r(\varphi)}\integ{a}{b}{|N( x(t))|}{t} \leq C
\end{eqnarray*}
\item[(P2.5)] $w_H:\RR\to\RR$ given by 
$$
w_H(h):= \sup_{t\in[a,b]}\integ{a}{b}{|\widetilde{H}(s+h,t)-\widetilde{H}(s,t)|}{s},
$$
satisfies
$$
\lim\limits_{h\to 0}w_H(h)=0.
$$
\end{itemize}

Let us notice the the assumptions (P1) and (P2.5) are the hypothesis of the extension of the product integration method to $L^1$ in the linear case (see \cite{ahues1}).

\begin{prop}\label{U}
If the properties {\rm(P1)} and {\rm(P2)} are verified, then $U$ is defined from $\Omega_r(\varphi)$ into $L^1([a,b],\CC)$, and it is a continuous compact  operator.
\end{prop}
\begin{proof}
$\forall x\in \Omega_r(\varphi)$, from the second order  Taylor expansion with integral remainder we get
\begin{eqnarray*}
\Vert U(x)\Vert&\le&\Vert K(x)\Vert+\Vert y\Vert\le\Vert K(\varphi)\Vert +\Vert K'(\varphi)(x-\varphi)\Vert\\
&&+\frac{1}{2}\sup_{u\in \Omega_r(\varphi)}\Vert K''(u)\Vert \Vert x-\varphi\Vert^2 + \|y\|,
\end{eqnarray*}
so that
\begin{align}\label{majU}
\Vert U(x)\Vert \leq \Vert K(\varphi)\Vert +rM_1 +\frac{1}{2}r^2 M_2 +\|y\|.
\end{align}
This proves that $U$ is defined from $\Omega_r(\varphi)$ into $L^1([a,b],\CC)$. \\
Let $B$ be a subset of $\Omega_r(\varphi)$ and define  $W:=\widetilde{U}(B)$, 
where
\begin{eqnarray*}
\widetilde{U}(x)(s) 
&:=&\left\{
\begin{array}{ll}
U(x)(s)&\textrm{for $s\in [a,b]$},\\\\
0 &\textrm{for $s\notin [a,b]$}.
\end{array}\right.
\end{eqnarray*}
From (\ref{majU}), $W$ is bounded in $L^1(\RR,\CC)$.
Let us prove that 
$$ 
\lim_{h\to 0}\Vert\tau_h f-f\Vert=0\quad\mbox{ uniformly in } f\in W.
$$
\begin{eqnarray*}
\Vert \tau_h\widetilde{U}(x)-\widetilde{U}(x)\Vert  
&\leq& c_L\integ{a}{b}{\integ{a}{b}{|\widetilde{H}(s+h,t)-\widetilde{H}(s,t)\Vert N(x(t))|}{t}}{s}\\
&&+\integ{a}{b}{\integ{a}{b}{\big|L(s+h,t)-L(s,t)\big\Vert \widetilde{H}(s,t)\Vert N(x(t))|}{s}}{t}\\
&\leq& c_Lw_H(h)C+2w_2(L,h)\Vert A_0(x)\Vert\\
&\leq& c_Lw_H(h)C +2w_2(L,h)M_0.
\end{eqnarray*}
Hence 
$$
\lim_{h\to 0}\sup_{x\in \Omega_r(\varphi)}\Vert \tau_h\widetilde{U}(x)-\widetilde{U}(x)\Vert =0.
$$
By the Kolmogorov-Fréchet-Riesz theorem, $U(B)=W|_{[a,b]}$ has a compact closure, thus $U$ is compact. As $K$ is  continuous, $U$ is continuous.
\end{proof}

\begin{prop}\label{cp}
The sequence $(U_n)_{n\geq 1}$ satisfies $U_n\LL\limits^{p} U \mbox{ on } \Omega_r(\varphi).$  
\end{prop}
\begin{proof}
For all $x\in \Omega_r(\varphi)$,
\begin{eqnarray*}
\Vert U_n(x)-U(x)\Vert  
&\leq&\int^b_a\big|\int^b_a\big([L(s,t)]_n-L(s,t)\big)H(s,t)N(\pi_n(x)(t))\big)dt\big|ds \\
&&+\int^b_a\big|\int^b_a H(s,t)L(s,t)\big(N(\pi_n(x)(t))-N(x(t))\big)dt\big|ds\\
&\leq& 2w_2(L,h_n)\Vert A_0(\pi_n(x))\Vert +\Vert K(\pi_n(x))-K(x)\Vert \\
&\leq& 2w_2(L,h_n)m_0+ \Vert K'(x)\Vert \Vert \pi_n(x)-x \Vert\\
&&+\frac{1}{2}\Vert \pi_n(x)-x\Vert^2 \sup_{v\in \Omega_r(\varphi)}\Vert K''(v)\Vert \\
&\leq& 2w_2(L,h_n)m_0+M_1 \Vert \pi_n(x)-x \Vert +\frac{1}{2} M_2 \Vert \pi_n(x)-x \Vert^2.
\end{eqnarray*}
Hence 
\begin{eqnarray}\label{U_n-U}
\Vert U_n(x)-U(x)\Vert  
 \leq 2w_2(L,h_n)m_0+M_1 \Vert \pi_n(x)-x \Vert +\frac{1}{2} M_2 \Vert \pi_n(x)-x \Vert^2.
\end{eqnarray}
As $\pi_n\LL\limits^{p} I$, $(U_n)_{n\geq 1}$ is pointwise convergent to $U$.
\end{proof}

\begin{prop}\label{thorKn}
If the properties {\rm(P1)} and {\rm(P2)} are verified, then $U_n$ is a continuous compact operator from $\Omega_r(\varphi)$ into $L^1([a,b],\CC)$, and $(U_n)_{n\geq 1}$ is a collectively compact sequence.
\end{prop}
\begin{proof}
$U_n$ is continuous on $\Omega_r(\varphi)$ because $K_n$ is Fr\'echet-differentiable. 

Let us prove that $(U_n)_{n\geq 1}$ is collectively compact. This is equivalent to prove that 
$$
F:=\bigcup\limits_{n\geq 1}U_n(B)
$$ 
is relatively compact for all bounded  subset $B$ of $\Omega_r(\varphi)$.We define the subset $E$ by 
$$
E:=\bigcup\limits_{n\geq 1}\widetilde{U}_n(B),
$$
where
\begin{eqnarray*}
\widetilde{U}_n(x)(s)&:= \left\{
\begin{array}{ll}
U_n(x)(s) &\textrm{for $s\in [a,b]$},\\\\
0 &\textrm{for $s\notin [a,b]$}.
\end{array}\right.
\end{eqnarray*}
Then
\begin{eqnarray*}
\Vert \widetilde{U}_n(x)\Vert &\leq& \Vert K_n(x)-K_n(\varphi)\Vert + \Vert  K_n(\varphi)\Vert +\Vert y\Vert \\
&\leq& \Vert K_n'(\varphi)(x-\varphi)+
\integ{0}{1} {(1-t)K_n''(\varphi+t(x-\varphi))(x-\varphi,x-\varphi)}{t} \Vert + \Vert  K_n(\varphi)\Vert +\Vert y\Vert \\
&\leq&
\Vert K_n'(\varphi)(x-\varphi)\Vert+
\frac{1}{2} \sup_{v\in \Omega_r(\varphi)}\Vert K_n''(v)\Vert \Vert x-\varphi  \Vert^2 + \Vert  K_n(\varphi)\Vert +\Vert y\Vert \\
&\leq&
r \Vert K_n'(\varphi)\Vert+
\frac{r^2}{2} C_2  + \Vert  K_n(\varphi)\Vert +\Vert y\Vert \\
&\leq&
r C_1+
\frac{r^2}{2} C_2  + c_L m_0 +\Vert y\Vert,
\end{eqnarray*}
hence $E$ is uniformly bounded.

For all $x\in \Omega_r(\varphi)$,
\begin{eqnarray*}
\Vert \tau_h\widetilde{U}_n(x)-\widetilde{U}_n(x)\Vert 
&=&\int^b_a\big|\int^b_a\big[\widetilde{H}(s+h,t)[\widetilde{L}(s+h,t)]_n\\\\
&&-\widetilde{H}(s,t)[\widetilde{L}(s,t)]_n\big]N(\pi_n(x)(t))dt\big|ds\\\\
&\leq& c_L\integ{a}{b}{\integ{a}{b}{|\widetilde{H}(s+h,t)-\widetilde{H}(s,t)\Vert N(\pi_n(x)(t))|}{t}}{s}\\\\
&&+\integ{a}{b}{\integ{a}{b}{\big|[\widetilde{L}(s+h,t)]_n-[\widetilde{L}(s,t)]_n\big\Vert \widetilde{H}(s,t)\Vert N(\pi_n(x)(t))|}{t}}{s}\\\\
&\leq& c_LMw_H(h)+2w_2(L,h)m_0.
\end{eqnarray*}
Thus, by the Kolmogorov-Fréchet-Riesz theorem, $F:=E\big|_{[a,b]}$ has a compact closure, and $(U_n)_{n\geq1}$ is collectively compact.
\end{proof}

\begin{theorem}\label{theogen}
Assume that $1$ is not an eigenvalue of $U'(\varphi)$, and that {\rm(P1)} and {\rm(P2)} are verified. Then $\varphi$ is an isolated fixed point of $U$. Moreover there are $\epsilon\in]0,r[$ and $n_\epsilon>0$ such that, for every $n\geq n_\epsilon$, $U_n$ has a unique fixed point $\varphi_n$ in $\Omega_\epsilon(\varphi)$. Also, there is a constant $\gamma>0$ such that, for $n\geq n_\epsilon$,
\begin{equation}
\Vert \varphi-\varphi_n\Vert \leq \gamma(2w_2(L,h_n)m_0+2M_1 w_1(\varphi,h_n) +2M_2w_1^2(\varphi,h_n))
\end{equation}
\end{theorem}
\begin{proof}
By Proposition \ref{U}, Proposition \ref{cp}, and Proposition \ref{thorKn}, conditions (H1) to (H4) in Theorem \ref{atkinson} are satisfied. The estimation is obtained by (\ref{estatkin}), (\ref{U_n-U}) and (\ref{pin}).\end{proof}

\section{Implementation and numerical evidence}
  
The approximate solution  is the exact solution of the equation
\begin{equation}\label{kn}
K_n(\varphi_n)-y=\varphi_n,
\end{equation}
where
\begin{eqnarray*}
K_n(\varphi_n)(s) 
&:=&\sum_{j=1}^{n}w_{n,j}(s)N(c_{n,j}),\\
w_{n,j}(s)&:=&\int_{t_{n,j-1}}^{t_{n,j}}H(s,t)[L(s,t)]_ndt,\\
c_{n,j}&:=&\frac{1}{h_n}\int_{t_{n,j-1}}^{t_{n,j}}\varphi_n(s)ds. 
\end{eqnarray*}

For $ i=1,\dots,n$, integrating (\ref{kn}) over $[t_{n,i-1},t_{n,i}]$ and dividing by $h_n$, we obtain the following nonlinear system
$$
\sum_{j=1}^{n}\frac{1}{h_n}\int_{t_{n,i-1}}^{t_{n,i}}w_{n,j}(s)dsN(c_{n,j})-\frac{1}{h_n}\int_{t_{n,i-1}}^{t_{n,i}}\varphi_n(s)ds=\frac{1}{h_n}\int_{t_{n,i-1}}^{t_{n,i}}y(s)ds
$$
for $i=1,\dots,n$.

Set
\begin{eqnarray*}
{\sf Y}_n(i)&:=&\displaystyle\frac{1}{h_n}\int_{t_{n,i-1}}^{t_{n,i}}y(s)ds,\\
{\sf A}_n(i,j)&:=&=\displaystyle\frac{1}{h_n}\int_{t_{n,i-1}}^{t_{n,i}}w_{n,j}(s)ds,\\
{\sf C}_n&:=&\left[\begin{array}{c}c_{n,1}\\
\vdots\\
c_{n,n}
\end{array}\right].
\end{eqnarray*}

We can rewite the nonlinear system in the matrix form 
\begin{eqnarray}\label{nonlinsys}
{\sf A}_nN({\sf C}_n)-{\sf C}_n={\sf Y}_n,
\end{eqnarray}
where
$$
N({\sf C}_n):=\left[\begin{array}{c}N(c_{n,1})\\
\vdots\\
N(c_{n,n})
\end{array}\right].
$$

Let ${\sf F}_n:\CC^{n{\times}1}\to\CC^{n{\times}1}$ be the operator defined by
$$
{\sf F}_n({\sf X}):={\sf A}_nN({\sf X})-{\sf X}-{\sf Y}_n,\quad{\sf X}\in\CC^{n{\times}1}.
$$
Newton's method will be applied to solve numerically the nonlinear problem 
$$
{\sf F}_n({\sf C}_n)={\sf 0}.
$$
Tables 1, 2 and 3 show the convergence of Newton's sequence for $n=10$ and $n=100$. The asumptions of Theorem \ref{theogen} are satisfied since $N$, $N'$ and $N''$ are bounded.

\bigskip

\noindent{\bf Example 1}  

For all $s,t\in [0,1]$, and $u\in \RR$,
\begin{eqnarray*}
L(s,t)&:=&1,\\
H(s,t)&:=&-\log(|s-t|),\\
N(u)&:=&\sin(\pi u)\mbox{ or }\sin(2\pi u).
\end{eqnarray*}
We chose 
$$
\varphi(s):=1,\quad s\in [0,1],
$$ 
to be the exact solution, so that 
$$
y(s):=-1,\quad s\in [a,b].
$$
\begin{table}
\begin{center}
\begin{tabular}{c|c|c}
$k$ & $\dfrac{\Vert C_{10}^{(k)}-C_{10}\Vert}{\Vert C_{10}\Vert}$ & $\dfrac{\Vert C_{100}^{(k)}-C_{100}\Vert}{\Vert C_{100}\Vert}$ \\ 
& &   \\
\hline 
& &   \\
1 &    3.5e-01    &  3.5e-01 \\ 

2 &  1.9e-01  & 1.9e-01 \\ 

3 &  2.3e-02 & 2.3e-02 \\ 

4 & 4.9e-05    & 5.1e-05 \\ 
 
5 & 8.2e-13    & 9.4e-13 \\ 
 
6 & 2.9e-16 & 1.3e-15  \\ 
\end{tabular}
\end{center}
\caption{Relative errors for $N(u)=\sin(\pi u)$ in Example 1}
\end{table}
\begin{table}
\begin{center}
\begin{tabular}{c|c|c}
$k$ & $\dfrac{\Vert C_{10}^{(k)}-C_{10}\Vert}{\Vert C_{10}\Vert}$ & $\dfrac{\Vert C_{100}^{(k)}-C_{100}\Vert}{\Vert C_{100}\Vert}$ \\ 
& &   \\
\hline 
& &   \\
10 &    1.5e-01    &  1.4e-02 \\ 

11 &  1.5e-01  & 6.9e-03 \\ 

12 &  1.4e-01 & 2.3e-03 \\ 
 
13 & 1.5e-01    & 1.5e-04 \\ 
 
14 & 1.5e-01    & 4.2e-08 \\ 
 
15 & 1.5e-01 & 1.6e-13 \\ 

16 & 1.5e-01 & 7.3e-14
\end{tabular}
\end{center}
\caption{Relative errors for $N(u)=\sin(2\pi u)$ in Example 1}
\end{table}

\bigskip

\noindent{\bf Example 2}

For all $s,t\in [0,1]$, and $u\in \RR$,
\begin{eqnarray*}
L(s,t)&:=&1,\\
H(s,t)&:=&-\log(|s-t|),\\
N(u)&:=&\sin(\pi u),\\
\varphi(s)&:=&\left\{
\begin{array}{ll}
1 &\textrm{for $s\in [0,0.5]$},\\
2 &\textrm{for $s\in [0.5,1]$},
\end{array}\right.\\
y(s)&:=&\left\{
\begin{array}{ll}
-1 &\textrm{for $s\in [0,0.5]$},\\
-2 &\textrm{for $s\in [0.5,1]$}.
\end{array}\right.
\end{eqnarray*}
\begin{table}
\begin{center}
\begin{tabular}{c|c|c}
$k$ & $\dfrac{\Vert C_{10}^{(k)}-C_{10}\Vert}{\Vert C_{10}\Vert}$ & $\dfrac{\Vert C_{100}^{(k)}-C_{100}\Vert}{\Vert C_{100}\Vert}$ \\
 & &   \\
\hline 
& &   \\
1 &   5.7e-02     & 5.6e-02  \\                       
2 &  3.7e-02  &   2.1e-02\\ 

3 &  1.9e-02 & 1.6e-03 \\ 
 
4 & 1.6e-03 &  5.5e-07 \\ 
 
5 & 6.8e-07    & 1.9e-15 
\end{tabular} 
\end{center}
\caption{Relative errors of the Newton iterates in Example 2}
\end{table}

\bigskip

The accuracy of the approximation is limited by $n$ (see Table 2 for $n=10$), especially when $N(u):=\sin(2\pi u)$. In order to overcome this difficulty, we are working on an approach which consists in linearizing the nonlinear equation by a Newton-type method in infinte dimension, and then  applying the product integration method to the linear equations issued from the Newton's method. We expect that the accuracy will not be $n$-sensitive.  

\begin{center}
{\bf Acknowledgements}
\end{center}
The first and the third authors were partially supported by an Indo-French Centre for Applied Mathematics (IFCAM) project.


\begin{thebibliography}{4} 

\bibitem{ahues1}
M. Ahues, L. Grammont and H. Kaboul, An extension of the product integration method to $L^1$ with application to astrophysics, HAL Id : hal-01232086, version 1
 
\bibitem{anselone}
 P.M Anselone and R. Ansorge, compactness principles in nonlinear operator approximation theory, {\em Numer. Funct. Anal. and  Optimiz.} 1 (6) 586-618 (1979) 
  
\bibitem{anselone1}
 P.M Anselone, {\em Collectively Compact Operator Approximation Theory and Application to Integral Equations}, Prentice-Hall, Inc., Englewood Cliffs, New Jersey, 1971.
   
\bibitem{ansorge}
  R. Ansorge, Convergence of Discretizations of Nonlinear Problems. A General Approach, {\em  Z. angew. Math. Mech.} 73 (1993) 10,239-253  
 
\bibitem{atkinson}
K.~E. Atkinson, {\em The Numerical Solution of Integral Equations of the Second Kind}, Cambridge University Press, 1997.

\bibitem{atkinson1}
K.E. Atkinson,   The Numerical Evaluation of Fixed Points for Completely  Continuos Operators, {\em SIAM Journal of Numerical Analysis}Vol. 10, No. 5, October 1973.

\bibitem{atkinson2}
K.E. Atkinson,   A Survey of Numerical Methods for Solving Nonlinear Integral Equation, {\em Journal of integral equations and applications}, Vol. 4, No. 1, Winter 1992.

\bibitem{bertram} B. Bertram, On the product integration method
for solving singular integral equations
in scattering theory, {\em Journal of Computational and Applied Mathematics} 25 (1989) 79-92

\bibitem{brezis}
H. Brezis,   {\em Functional Analysis, Sobolev Spaces and Partial Differential Equations}, Springer-Verlag New York, 2010.

\bibitem{brunner}
H. Brunner   and T. Tang, Polynomial Spline Collocation  Methods
For the nonlinear Basset equation, {\em Computers Math. Applic.} Vol. 18, No. 5, pp. 449-457, 1989

\bibitem{cameron}
R.F Cameron and S. Mckee
Product integration methods for second-kind Abel integral equations,{\em Journal of Computational and Applied 
Mathematics} 11 (1984) 1-10  
 

\bibitem{diogo}
T. Diogo, N. B. Franco and  P. Lima, High Order Product Integration Method for a Volterra Integral Equation with Logarithmic singular Kernel, {\em Communications on Pure and Applied Analysis} Vo. 3, Issue 2, Pages: 217-235,  June 2004 

\bibitem{weaklysingular}
H. Kaneko, R. D.Noren and Y. Xu,   Numerical Solution for Weakly Singular Hammerstein Equations and Their Superconvergence, {\em Journal of integral equations and applications}, Vol. 4, No. 3, Summer 1992.

\bibitem{kra1}
M.~A. Krasnoselskii, G. Vainikko, P.~P. Zabreiko, Ya.~B. Rutitskii and V.~Ya. Stetsenko, {\em Approximate Solution of Operator Equations,} Noordhoff, Groningen, the Nederlands, 1972.

\bibitem{mckee1}
Tao Tang and Sean McKee, 
   Product integration method for an  integral equation with logarithmic singular kernel, {\em Applied 
Numerical 
Mathematics} 9 (1992) pp 253-266 

\end{thebibliography}
\end{document}